\documentclass{article}

\usepackage[T2C]{fontenc}
\usepackage[cp866]{inputenc}
\usepackage[russian,english]{babel}
\usepackage[tbtags]{amsmath}
\usepackage{amsfonts,amssymb,mathrsfs,amscd,longtable}
\usepackage{mathrsfs}
\usepackage[izv]{sks2e-tr}
\usepackage[all]{xy}

\usepackage{hypbmsec}
\usepackage{mathtext}

\def\bad{\spaceskip=0.33emplus0.6emminus0.15em\immediate\write5{\string\bad}}

\theoremstyle{plain}
\newtheorem{theorem}{Theorem}
\newtheorem{lemma}{Lemma}

\theoremstyle{definition}

\newtheorem{proof}{Proof}
\newtheorem{remark}{Remark}

\begin{document}

\hyphenation{англ назв МИАН Докл МЦНМО Фазис МАИК}

\title{On the Karatsuba divisor problem}

\doi{10.4213/im9270}

\author{V.~V.~Iudelevich}[Vitalii~V.~Iudelevich]%
\address{Lomonosov Moscow State University}

\email{vitaliiyudelevich@mail.ru}

\date{02.10.2021}
\udk{511.337}

\maketitle

\markright{On the Karatsuba divisor problem}

\begin{fulltext}

\footnotetext[0]{The paper was published in the journal Izv. RAN (2022), \textbf{86}(5), pp. 169-196.}

\begin{abstract}
We obtain an upper bound for the sum $\Phi_a(x) = \sum_{p\le x}1/(\tau(p+a))$, where $\tau(n)$ is the divisor function, $a\ge 1$ is a fixed integer, and $p$ run through primes up to $x$.

Bibliography: 14 titles.
\end{abstract}



\section*{Introduction}
In 2004, A.~A. Karatsuba in his seminar ``Analytic number theory and applications'' suggested the following problem: find the asymptotic formula for the sum
$$
\Phi_a(x) = \sum_{p \le x}\frac{1}{\tau(p+a)},\qquad x\to +\infty,
$$
where $\tau(n)$ denotes divisor function, $a$ is a fixed integer, and the summation is taken over primes not exceeding $x$.

This problem is a result of discussions between A.\,A.~Karatsuba and V.\,I.~Arnold. Note that it  contains features of the following two problems from analytic number theory. The first one is Titchmarsh divisor problem. It is asking about the asymptotic behavior of the sum
$$
F_a(x) = \sum_{p\le x}\tau(p+a).
$$
In 1930, Titchmarsh \cite{1} proved the estimate
$$
\sum_{p\le x}\tau(p-1) = O(x).
$$
He also proved under Generalized Riemann Hypothesis that
$$
\sum_{p\le x}\tau(p-1)\sim cx,\qquad c = \frac{\zeta(2)\zeta(3)}{\zeta(6)}.
$$
The first unconditional result was obtained by Yu.~V.~Linnik \cite{2}. He proved with the dispersion method that
\begin{gather*}
\sum_{p\le x}\tau(p-1) = cx +R(x),\qquad c=\frac{\zeta(2)\zeta(3)}{\zeta(6)},
\\
R(x) \ll \frac{x}{(\ln x)^{\alpha}},\qquad 0 < \alpha < 1.
\end{gather*}
Subsequently, this result was refined by many authors (see\ \cite{3}--\cite{7}).
The second problem is to find the asymptotic formula for the sum 
$$
T(x) = \sum_{n\le x}\frac{1}{\tau(n)}.
$$
In 1916, S. Ramanujan \cite{8} proved that
$$
T(x) = \frac{x}{\sqrt{\ln x}}\biggl(a_0 +\frac{a_1}{\ln x}+\dots+\frac{a_n}{(\ln x)^n}+O_n\biggl(\frac{1}{(\ln x)^{n+1}}\biggr)\biggr),
$$
where $a_n$ are some constants; in particular,
$$
a_0 = \frac{1}{\sqrt{\pi}}\prod_{p}\sqrt{p^2-p}\,\ln{\frac{p}{p-1}}.
$$
The aim of this paper is to find an upper bound for the sum $\Phi_a(x)$.

Since
$$
\frac{1}{x}\sum_{p\le x}1\asymp \frac{1}{\ln x},
$$
and
$$
\frac{1}{x}\sum_{n\le x}\frac{1}{\tau(n+a)}\asymp\frac{1}{\sqrt{\ln x}},
$$
then it is natural to expect that
$$
\frac{1}{x}\Phi_a(x)\asymp\frac{1}{(\ln x)\sqrt{\ln x}} = \frac{1}{(\ln x)^{3/2}}.
$$
In this paper, we obtain the upper bound
$$
\Phi_a(x)\ll_a\frac{x}{(\ln x)^{3/2}}.
$$
More precisely, we prove the following theorem.

\begin{theorem}
\label{t1}
Let $a\ge 1$ be a fixed integer. Then we have
$$
\sum_{p\le x}\frac{1}{\tau(p+a)}\le{4 K(a)}\frac{x}{(\ln x)^{3/2}}+O\biggl(\frac{x \ln\ln x}{(\ln x)^{5/2}}\biggr),
$$
where
\begin{gather*}
K(a) = K\beta(a),\qquad K= \frac{1}{\sqrt{\pi}}\prod_{p}\sqrt{\frac{p}{p-1}}\, \biggl(p\ln\frac{p}{p-1} -\frac{1}{p-1}\biggr),
\\
\beta(a) = \prod_{p\,|\,a}\biggl(1+\frac{1}{p(p-1)\ln(p/(p-1))-1}\biggr).
\end{gather*}
\end{theorem}

Thus, the upper bound coincides with the conjectural order of magnitude of $\Phi_a(x)$.

In this paper we rely on Selberg's sieve method. The essence of this method is the following. We need to estimate the sum
$$
S(\mathcal{A},z) = \sum_{(n,P(z)) = 1}a_n,
$$
where $\mathcal{A} = (a_n)$ is a given sequence of non-negative real numbers and $P(z) = \prod_{p\le z}p$ is the product of all primes up to $z$. 
Since
$$
\sum_{d\,|\,n}\mu(d) = \begin{cases}
1, &\text{if }n = 1,
\\
0, &\text{otherwise},
\end{cases}
$$
it follows that
$$
S(\mathcal{A},z) = \sum_{n}a_n\sum_{d\,|\,(n,P(z))}\mu(d).
$$
Let us introduce arbitrary numbers $\rho_d$ (here $d \le z$, $d\,|\,P(z)$), for which $\rho_1 = 1$. Then $\sum_{d\,|\,n}\mu(d)\,{\le} \bigl( \sum_{d\,|\,n}\rho_d\bigr)^2$ 
for all $n\ge 1$. Thus, we obtain
$$
S(\mathcal{A},z)\le \sum_{n}a_n\biggl(\sum_{d\,|\,(n,P(z))} \rho_d\biggr)^2 = \sum_{d_1,d_2\,|\,P(z)}\rho_{d_1}\rho_{d_2}A_{[d_1,d_2]},
$$
where
$$
A_d = \sum_{n \equiv 0\,(\operatorname{mod} d)}a_n.
$$
Further, assume that for considered $d$ we have $A_d = Xg(d)+r_d$, where $g(d)$ is multiplicative function, $X$ is independent of $d$, and $r_d$ is small in average. 
The coefficients $\rho_d$ are chosen to minimize the quadratic form.
\begin{equation}
\label{eq1}
B = \sum_{d_1, d_2\,|\,P_a(z)}\rho_{d_1}\rho_{d_2}g([d_1,d_2]).
\end{equation}
In our case, $A_d$ has a more complicated form
$$
A_d = X_0 g_0(d)+ X_1 g_1(d)+\dots+X_{m} g_m(d),
$$
where $g_k(d)$ are some functions (generally speaking, $g_k(d)$ are not multiplicative for $k\geqslant 1$), and the values $X_k$ are independent of $d$. The coefficients $\rho_d$ are constructed only from the first quadratic form \eqref{eq1} corresponding to the function $g\,{=}\,g_0$. 
In this case we observe the following effect. These coefficients simultaneously minimize the remaining quadratic forms corresponding to the functions $g_k$, at least in terms of the order of magnitude.

Note that the solution of Titchmarsh divisor problem is based on Bombieri-Vinogradov theorem. 
In our paper we essentially use the analogue of Bombieri-Vinogradov theorem, 
obtained by M.~A.~Korolev (see \ \cite[Lemma~13]{9}).
We observe that the methods of this paper can also be applied for other sums. For instance, one can show that if
$$
T_a(x) = \sum_{\substack{p \le x \\ p,\,p+2\text{ primes}}}\frac{1}{\tau{(p+a)}},
$$
where the summation is over twin primes, then
$$
T_a(x)\le \frac{c(a) x}{(\ln x)^{5/2}}(1+o(1)).
$$
\textbf{Acknowledgements.} The work was supported by the Theoretical Physics and Mathematics
Advancement Foundation ``BASIS''.

\section{Auxiliary results}
\label{s1}

\begin{lemma}
\label{l1}
Determine the coefficients $c_k$ from the expansion
\begin{equation}
\label{eq2}
\frac{x}{-(\ln(1-x))} = \sum_{k = 0}^{+\infty}c_kx^k,\qquad |x|<1.
\end{equation}
Then $c_0 = 1$, and $|c_k|\le 1$ for $k\ge 1$.
\end{lemma}

\begin{proof}
This immediately follows from the well-known identity (see \cite{10}):
\begin{equation}
\label{eq3}
\sum_{k=1}^{+\infty}|c_k| = 1.
\end{equation}
The proof of the lemma is complete.
\end{proof}

\begin{remark}
The numbers
$$
G_n = (-1)^n c_n,\qquad n\ge 1,
$$
 are called Gregory coefficients. For more precise estimates of the Gregory coefficients, see ~\cite{10}.
\end{remark}

\begin{lemma}
\label{l2}
For $0<\varepsilon<1$ and $d\ge 1$ we have
\begin{equation}
\label{eq4}
\sum_{p\,|\,d}\frac{1}{p^{1-\varepsilon}}\le \frac{2\omega(d)^{\varepsilon}}{\varepsilon},
\end{equation}
where $\omega(d)$ is the prime divisor function.
\end{lemma}

\begin{proof}
This is clear if $d$ is prime power or $d = 1$. Assume that $\omega(d)\ge 2$, then for $X\ge 2$ we have
\begin{align*}
\sum_{p\,|\,d}\frac{1}{p^{1-\varepsilon}} &= \biggl( \sum_{p\,|\,d,\,p \le X} + \sum_{p\,|\,d,\,p > X}\biggr) \frac{1}{p^{1-\varepsilon}}\le \sum_{p\,|\,d,\,p \le X}\frac{1}{p^{1-\varepsilon}} + \frac{1}{X^{1-\varepsilon}}\sum_{p\,|\,d}1
\\
&\le \sum_{2\le n \le x}\frac{1}{n^{1-\varepsilon}} + \frac{\omega(d)}{X^{1-\varepsilon}}\le \int_1^X\frac{du}{u^{1-\varepsilon}} + \frac{\omega(d)}{X^{1-\varepsilon}}\le \frac{X^\varepsilon}{\varepsilon}+\frac{\omega(d)}{X^{1-\varepsilon}}.
\end{align*}
Choosing $X = \omega(d)\ge 2$ we conclude the proof.
\end{proof}

We define the function $G_d(s)$ by the equality
\begin{equation}
\label{eq5}
G_d(s) = H(s)J_d(s).
\end{equation}
Here
\begin{gather}
H(s)=\frac{1}{s}{\sqrt{\zeta(s)(s-1)}\, \prod_{p}\sqrt{p^{2s}-p^s}\, \ln\frac{p^s}{p^s-1}},
\nonumber
\\
\label{eq6}
J_d(s)=\prod_{p\,|\,d}\biggl(p^s\ln\frac{p^s}{p^s-1}\biggr)^{-1},
\end{gather}
where we choose the principal branches of $\sqrt{z}$ and $\ln z$. It is well-known (see \ \cite[chapter IV, \S\,3, Theorem 1]{11}) that the Riemann zeta-function has no zeroes in the region
\begin{equation}
\label{eq7}
\sigma \ge 1 - \frac{c_0}{(\ln |t|)^{2/3}(\ln \ln |t|)^{1/3}},\qquad t\ge t_0,
\end{equation}
for some $c_0>0$. Thus, the function $G_d(s)$ is regular in the domain \eqref{eq7}.

\begin{lemma}
\label{l3}
Set
\begin{equation}
\label{eq8}
\varepsilon_d = \bigl( 3\ln(\omega(d)+2)\bigr)^{-1}.
\end{equation}
If $\operatorname{Re} s\ge 1-\varepsilon_d/2$ and $l\ge 0$, then for the $l$-th derivative of $J_d(s)$, defined in \eqref{eq6}, one has
$$
J^{(l)}_d(s)\ll_l(\omega(d)+2)^{10}.
$$
\end{lemma}

\begin{proof}
We have,
$$
J_d(s) = \prod_{p\,|\,d}\frac{p^{-s}}{-\ln(1-p^{-s})} = \prod_{p\,|\,d}(1+c_1p^{-s}+c_2p^{-2s}+\cdots),
$$
where the coefficients $c_k$ are defined in \eqref{eq2}. Hence, we obtain
$$
J_d(s) = \sum_{\delta\,|\,d^\infty}\frac{j(\delta)}{\delta^s},
$$
where the symbol $\delta\,|\, d^\infty$ means the summation over those numbers whose prime factors divide $d$, and $j(\delta)$ is the multiplicative function, for which $j(p^k) = c_k$. Thus, we obtain
$$
J_d^{(l)}(s) = \sum_{\delta\,|\,d^{\infty}}\frac{j(\delta)(-\ln\delta)^l}{\delta^s}\ll \sum_{\delta\,|\,d^{\infty}}\frac{(\ln\delta)^l}{\delta^\sigma},
$$
where $s = \sigma+it$. Note that
$$
\frac{(\ln \delta)^l}{\delta^\varepsilon}\le \biggl(\frac{l}{e}\biggr)^l\frac{1}{\varepsilon^l},
$$
for $\varepsilon > 0$, $l\ge 1$ and $\delta \ge 1$. Hence, choosing $\varepsilon = \varepsilon_d/2$ and using inequality $\sigma\ge 1-\varepsilon_d/2$, we find that
$$
J_d^{(l)}(s)\ll_l\frac{1}{\varepsilon_d^l}\prod_{p\,|\,d}\biggl(1+\frac{1}{p^{1-\varepsilon_d}} +\frac{1}{p^{2(1-\varepsilon_d)}}+\cdots\biggr).
$$
Note that the obtained inequality is also holds for $l = 0$. Since
$$
2(1-\varepsilon_d)\ge 2\biggl( 1-\frac{1}{3\ln 2}\biggr)>1,
$$
it follows from Lemma \ref{l2} that
\begin{align*}
J_d^{(l)}(s) &\ll\frac{1}{\varepsilon_d^l}\prod_{p\,|\,d} \biggl(1+\frac{1}{p^{1-\varepsilon_d}}\biggr)\le \frac{1}{\varepsilon_d^l}\exp\biggl(\sum_{p\,|\,d}\frac{1}{p^{1-\varepsilon_d}}\biggr) \ll\frac{1}{\varepsilon_d^l}\exp\biggl(\frac{2\omega(d)^{\varepsilon_d}}{\varepsilon_d}\biggr)
\\
&\ll_l(\omega(d)+2)^{6\exp(1/3)}\bigl(\ln(\omega(d)+2)\bigr)^l\ll_l\bigl(\omega(d)+2\bigr)^{10}.
\end{align*}
This concludes the proof.
\end{proof}

\begin{lemma}
\label{l4}
Let $d\le x$ be an integer and $m\ge 0$ be a fixed number, then
\begin{equation}
\label{eq9}
\sum_{\substack{k\le x\\(k,d)=1}}\frac{1}{\tau(k)}=\frac{x}{\sqrt{\pi\ln x}}\sum_{k=0}^m\frac{(-1)^k\binom{2k}{k}}{4^k}\,\frac{ G_d^{(k)}(1)}{(\ln x)^k}+R_m(x;d),
\end{equation}
where $G_d(s)$ is defined in \eqref{eq5}, and
$$
R_m(x;d)\ll_m \kappa(d)\frac{x}{(\ln x)^{m+3/2}},\qquad \kappa(d) = (\omega(d)+2)^{10}.
$$
\end{lemma}

\begin{proof}
Consider the generating series
$$
F_d(s)=\sum_{\substack{n=1\\(n,d)=1}}^{+\infty}\frac{1}{\tau(n)}\, n^{-s}.
$$
Then we have
\begin{align*}
F_d(s) &=\prod_{p\,\nmid\, d} \biggl(1+\frac{1}{2}p^{-s}+\frac{1}{3}p^{-2s}+\cdots\biggr)=\prod_{p\,\nmid\, d} p^s\ln\frac{p^s}{p^s-1}
\\
&=\frac{\prod_{p}\bigl(p^s \ln(p^s/(p^s-1))(1-p^{-s})^{1/2} (1-p^{-s})^{-1/2}\bigr)}{\prod_{p\,|\,d}\bigl(p^s\ln(p^s/(p^s-1))\bigr)}
\\
&=\frac{\sqrt{\zeta(s)}\,\prod_{p}\bigl(\sqrt{p^{2s}-p^s}\, \ln(p^s/(p^s-1))\bigr)}{\prod_{p\,|\,d}\bigl(p^s\ln(p^s/(p^s-1))\bigr)},
\end{align*}
where $\sqrt{z}>0$ for $z>0$. Using Perron's formula (see\ \cite[chapter IV, \S\,1, Theorem~1]{12}) for $T, x \ge 2$ and $b=1+1/\ln x$, we obtain
$$
\sum_{\substack{k\le x\\(k,d)=1}}\frac{1}{\tau(k)}=j+O\biggl(\frac{x\ln x}{T}\biggr),
$$
where
\begin{gather*}
j=\frac{1}{2\pi i}\int_{b-iT}^{b+iT}F_d(s)\frac{x^s}{s}\, ds=\frac{1}{2\pi i}\int_{b-iT}^{b+iT}\frac{G_d(s)x^s}{\sqrt{s-1}}\, ds,
\\
G_d(s)\,{=}\,\frac{1}{s}\sqrt{\zeta(s)(s\,{-}\,1)}\prod_{p} \biggl(\sqrt{p^{2s}\,{-}\,p^s}\ln\frac{p^s}{p^s\,{-}\,1}\biggr) \prod_{p\,|\,d}\biggl(p^s\ln\frac{p^s}{p^s\,{-}\,1}\biggr)^{-1}
{=}\,H(s) J_d(s).
\end{gather*}
Put
$$
a=1-{c_0}{(\ln T)^{-2/3}}(\ln\ln T)^{-1/3},
$$
where $c_0$ is chosen as in~\eqref{eq7}. Consider the rectangle $\Gamma$ at the vertices $a\pm iT$, $b\pm iT$ with horizontal cut going straight from the point~$s = a$ to the~point~$s = 1$. Then by the Cauchy's theorem
\begin{align*}
\frac{1}{2\pi i}\int_{\Gamma}\frac{G_d(s)x^s}{\sqrt{s-1}}\, ds &=\frac{1}{2\pi i}\biggl(\int_{b-iT}^{b+iT}+\int_{b+iT}^{a+iT}+\int_{a+iT}^{a+i0}
\\
&\qquad+\int_{a+i0}^{1+i0}+\int_{1-i0}^{a-i0}+\int_{a-i0}^{a-iT}+\int_{a-iT}^{b-iT}\biggr) \frac{G_d(s)x^s}{\sqrt{s-1}}\, ds
\\
&=j+j_1+j_2+j_3+j_4+j_5+j_6=0
\end{align*}
(where definition of the symbols $j_1, \dots, j_6$ is obvious), thus
$$
j=-(j_3+j_4)-j_1-j_2-j_5-j_6.
$$
Let us calculate $J=-(j_3+j_4)$. We have
\begin{align*}
j_3&=\frac{1}{2\pi i}\int_{a+i0}^{1+i0}\frac{G_d(s)x^s}{\sqrt{s-1}}\, ds=\frac{1}{2\pi i}\int_{a}^{1}\frac{G_d(\sigma)x^\sigma}{\sqrt{\sigma-1+i0}} \, d\sigma
\\
&=\frac{1}{2\pi i}\int_{0}^{1-a}\frac{G_d(1-u)x^{1-u}}{\sqrt{u}\, \sqrt{-1+i0}}\, du =-\frac{x}{2\pi}\int_{0}^{1-a}\frac{G_d(1-u)x^{-u}}{\sqrt{u}}\, du.
\end{align*}
Similarly, we obtain
$$
j_4=-\frac{x}{2\pi}\int_{0}^{1-a}\frac{G_d(1-u)x^{-u}}{\sqrt{u}}\, du,
$$
hence,
$$
J=\frac{x}{\pi}\int_{0}^{1-a}\frac{G_d(1-u)x^{-u}}{\sqrt{u}}\, du.
$$
From the Taylor formula with Lagrange's remainder, we obtain
$$
G_d(1-u)=\sum_{k=0}^m (-1)^k G_d^{(k)}(1)\frac{u^k}{k!} +O_m\bigl(G_d^{(m+1)}(\theta)u^{m+1}\bigr),\qquad a\le1-u\le\theta \le1.
$$
Hence,
\begin{align*}
J &=\frac{x}{\pi}\int_{0}^{1-a}\sum_{k=0}^m\frac{G_d^{(k)}(1)(-1)^k}{k!}\, \frac{u^k x^{-u}}{\sqrt{u}}\, du +O_m\biggl(x G_{m+1} \int_{0}^{1-a}u^{m+1/2}x^{-u}\, du\biggr)
\\
&=\frac{x}{\pi}\sum_{k=0}^m\frac{(-1)^k G_d^{(k)}(1)j_k(a)}{k!} +O_m\bigl(x G_{m+1} j_{m+1}(a) \bigr),
\end{align*}
where
\begin{gather*}
G_r = \max_{a\le\theta \le 1}|G_d^{(r)}(\theta)|,\qquad r\le m+1,
\\
j_k(a)\,{=} \int_{0}^{1-a}u^{k-1/2}x^{-u}\, du \,{=} \int_{0}^{+\infty}u^{k-1/2}x^{-u}\, du- \int_{1-a}^{+\infty}u^{k-1/2}x^{-u}\, du \,{=}\, J_k-r_k.
\end{gather*}
Further, we choose $T = e^{(\ln x)^\alpha}$, $\alpha>0$. 
Since $\omega(d)\ll \ln x$ for $x\ge x_0$, it follows that $a\ge 1-\varepsilon_d/2$, where the value $\varepsilon_d$ is defined in~\eqref{eq8}.

 Therefore, for $a\le \theta\le 1$ we have
$$
G_d^{(r)}(\theta) = \sum_{l=0}^r\binom{r}{l}H^{(r-l)}(\theta)J_d^{(l)}(\theta) \ll_m \sum_{l=0}^r|J_d^{(l)}(\theta)|\ll_m \bigl(\omega(d)+2\bigr)^{10},
$$
so that for $r\le m+1$ we get
\begin{equation}
\label{eq10}
G_r\ll_m \bigl(\omega(d)+2\bigr)^{10}.
\end{equation}
For the value $J_k$ we have
\begin{align*}
J_k &=\int_{0}^{+\infty}u^{k-1/2}x^{-u}\, du=\frac{1}{(\ln x)^{k+1/2}} \int_{0}^{+\infty}w^{k-1/2}e^{-w}\, dw
\\
&=\frac{\Gamma(k+1/2)}{(\ln x)^{k+1/2}}=\sqrt{\pi}\, \binom{2k}{k}\frac{k!}{4^k}\, \frac{1}{(\ln x)^{k+1/2}}.
\end{align*}
Further,
$$
r_k=\int_{1-a}^{+\infty}u^{k-1/2}x^{-u}\, du=\frac{1}{(\ln x)^{k+1/2}}\int_{(1-a)\ln x}^{+\infty}w^{k-1/2}e^{-w}\, dw.
$$
Using the estimate
$$
I_{k}(\lambda)=\int_{\lambda}^{+\infty}w^{k-1/2}e^{-w}\, dw\ll k!\, e^{-\lambda}\lambda^{k-1/2},\ \ (\lambda >1),
$$
which is obtained by iterated integration by parts, we find that
\begin{align*}
r_k &\ll \frac{k!}{(\ln x)^{k+1/2}}e^{(a-1)\ln x}(1-a)^{k-1/2}(\ln x)^{k-1/2}
\\
&= \frac{k!\, x^{a-1}c_0^{k-1/2}}{(\ln T)^{(2/3)(k-1/2)}(\ln\ln T)^{(1/3)(k-1/2)}\ln x},
\end{align*}
where the constant in the symbol $\ll$ is absolute. Thus,
$$
j_k(a)=\sqrt{\pi}\,\binom{2k}{k}\frac{k!}{4^k}\, \frac{1}{(\ln x)^{k+1/2}}+O\biggl(\frac{k!\, x^{a-1}c_0^{k-1/2}}{(\ln T)^{(2/3)(k-1/2)}(\ln\ln T)^{(1/3)(k-1/2)}\ln x} \biggr)
$$
and
\begin{align*}
J &=\frac{x}{\sqrt{\pi\ln x}}\sum_{k=0}^m\frac{(-1)^k\binom{2k}{k}}{4^k}\, \frac{G^{(k)}(1)}{(\ln x)^{k}}
\\
&\qquad+O_m\biggl(\frac{x^a}{\ln x}\sum_{k=0}^m\frac{|G^{(k)}(1)|}{(\ln T)^{(2/3)(k-1/2)}(\ln\ln T)^{(1/3)(k-1/2)}}\biggr)
\\
&\qquad+O_m\biggl( G_{m+1}\frac{x}{(\ln x)^{m+3/2}} \biggr).
\end{align*}
Finally, from the estimate \eqref{eq10}, we get
\begin{align*}
J &=\frac{x}{\sqrt{\pi\ln x}} \sum_{k=0}^m\frac{(-1)^k\binom{2k}{k}}{4^k}\, \frac{G^{(k)}(1)}{(\ln x)^{k}}
\\
&\qquad + O_m\biggl( \kappa_d\frac{x^a(\ln T)^{1/3}(\ln\ln T)^{1/6}}{\ln x} + \kappa_d\frac{x}{(\ln x)^{m+3/2}}\biggr).
\end{align*}
Now we estimate the integrals $j_1$, $j_2$, $j_5$, $j_6$. It is well-known that 
$$
\zeta(\sigma+it)\ll(\ln|t|)^{2/3},
$$
where $\sigma\ge 1-c/(\ln t)^{2/3}$, $c>0$, $|t|\ge 10$,
(see\ \cite[chapter~IV, \S\,2, p.~3, Theorem~2]{11}). Then using this estimate, Lemma \ref{l3}, and the inequality
$$
\prod_{p}\sqrt{p^{2s}-p^s}\, \ln\frac{p^s}{p^s-1} = \prod_{p}\biggl(1-\frac{1}{24p^{2s}}-\frac{1}{24p^{3s}}-\cdots \biggr) \ll 1,
$$
which holds for $3/4\le \operatorname{Re} s\le 2$, we find that
\begin{align*}
j_1 &=\frac{1}{2\pi i} \int_{b+iT}^{a+iT}\frac{1}{s}\sqrt{\zeta(s)}\, \prod_{p}\biggl(\sqrt{p^{2s}-p^s}\ln\frac{p^s}{p^s-1}\biggr) \prod_{p\,|\,d}\biggl(p^s\ln\frac{p^s}{p^s-1}\biggr)^{-1}x^s\, ds
\\
&\ll \bigl(\omega(d)+2\bigr)^{10}\,\frac{x(\ln T)^{1/3}}{T}.
\end{align*}
Similarly, we obtain
$$
j_6\ll\bigl(\omega(d)+2\bigr)^{10}\,\frac{x(\ln T)^{1/3}}{T}.
$$
Further,
\begin{align*}
j_2+j_5 &=\frac{1}{2\pi i} \int_{a+iT}^{a-iT}\frac{1}{s}\sqrt{\zeta(s)}\, \prod_{p}\biggl(\sqrt{p^{2s}-p^s}\,\ln\frac{p^s}{p^s-1}\biggr) \prod_{p\,|\,d}\biggl(p^s\ln\frac{p^s}{p^s-1}\biggr)^{-1}x^s\, ds
\\
&\ll\bigl(\omega(d)+2\bigr)^{10}(\ln T)^{1/3} x^a \int_{-T}^T\frac{dt}{\sqrt{a^2+t^2}}\ll \bigl(\omega(d)+2\bigr)^{10}(\ln T)^{4/3} x^a.
\end{align*}
Choosing $T=e^{(\ln x)^{3/5}}$ we conclude the proof. 
\end{proof}

\begin{remark}

Using Lemma \ref{l6} and some estimates following from inequality \eqref{eq4}, one can improve the dependence on $d$ in the remainder term in~\eqref{eq9}.
\end{remark}

Denote by $I(s)$ the logarithmic derivative of the function $J_d(s)$ defined in~\eqref{eq6}, then we have
\begin{equation}
\label{eq11}
\frac{J_{d}^{'}(s)}{J_d(s)} = I(s),
\end{equation}
where
\begin{equation}
\label{eq12}
I(s) = \sum_{p\,|\,d} f(s;p) = \sum_{p\,|\,d}\biggl\{\frac{\ln p}{(p^s-1)\ln(p^s/(p^s-1))}-\ln p\biggr\}.
\end{equation}

\begin{lemma}
\label{l5}
Let $f(s;p)$ be defined in~\eqref{eq12}, then for $m\ge 0 $ we have
$$
\frac{d^m}{ds^m}f(1;p)\ll_{m}\frac{(\ln p)^{m+1}}{p}.
$$
\end{lemma}

\begin{proof}
Let us define the sequence $d_k$ from the expansion
$$
\frac{-x}{(1-x)\ln(1-x)} = \sum_{k=0}^{+\infty}d_k x^k,\qquad |x|<1,
$$
then $d_0=1$ and for $k\ge 1$ we have $d_k = c_0+c_1+\dots+c_k$, where the coefficients $c_k$ are defined in~\eqref{eq2}. From \eqref{eq3} we find that $|d_k|\le 2$, where $k\ge 1$. Further, we have
$$
\frac{1}{(p^s-1)\ln(p^s/(p^s-1))} = \frac{p^{-s}}{(1-p^{-s})(-\ln(1-p^{-s}))} = \sum_{k = 0}^{+\infty}\frac{d_k}{p^{ks}}.
$$
Hence, since $d_1 = 1/2$, it follows that
$$
f(s;p) = \frac{\ln p}{2 p^s}+\sum_{k=2}^{+\infty}\frac{d_k \ln p}{p^{ks}}.
$$
Taking the $m$-th derivative, we get
$$
f^{(m)}(s;p) = (-1)^m\frac{(\ln p)^{m+1}}{2 p^s} + (-1)^m (\ln p)^{m+1}\sum_{k=2}^{+\infty}\frac{d_k k^m}{p^{ks}},\ m\geqslant 0.
$$
Then for $s=1$ inequality $|d_k|\le 2$ implies that
\begin{align*}
f^{(m)}(1;p) &= (-1)^m\frac{(\ln p)^{m+1}}{2 p} + (-1)^m (\ln p)^{m+1}\sum_{k=2}^{+\infty}\frac{d_k k^m}{p^{k}}
\\
&=(-1)^m\frac{(\ln p)^{m+1}}{2 p } +O_{m, \varepsilon}\biggl( \frac{1}{p^{2-\varepsilon}}\biggr)\ll_m\frac{(\ln p)^{m+1}}{p}.
\end{align*}
The proof of the lemma is complete.
\end{proof}

\begin{lemma}
\label{l6}
Let $J_d(s)$ be defined in~\eqref{eq6} and $I(s)$ be defined in~\eqref{eq12}, then for $l\ge 1$ the following representation holds
$$
J_{d}^{(l)}(s) = J_d(s) Q_l,
$$
where
$Q_l = Q_l(I, I', \dots, I^{(l-1)})$~ is a polynomial in $l$ variables of degree $l$ with integer coefficients.
\end{lemma}

\begin{proof}
We prove that $Q_l$ has the form
$$
Q_l = I^l + R_l,
$$
where
$$
R_l = R_l(I, I', \dots, I^{(l-1)})\in\mathbb{Z}[I, I', \dots, I^{(l-1)}]
$$
and $\deg R_l\le l-1$.

The proof is by induction on $l$. According to \eqref{eq11} for $l = 1$ we have $J_{d}'(s) = J_d(s)I$.
Hence, we obtain that $R_1 = 0$, so the lemma is true in this case.

Assume that the lemma is true for $l = r\ge 1$, let us prove it for $l = r+1$. We have
$$
J_d^{(r+1)}(s) = J_d(s)(I^{r+1} + IR_r + rI^{r-1}I' + R_{r}') = J_d(s)(I^{r+1} + R_{r+1}).
$$
Since $\deg R'_r\le r-1 $, $\deg I R_r\le r$ and $\deg I^{r-1}I' = r$, it follows that
$$
\deg R_{r+1}\le r.
$$
The claim follows.
\end{proof}

Let $g(n)$~ be a multiplicative function such that for any prime~$p$
\begin{equation}
\label{eq13}
0\le g(p) <1.
\end{equation}
Further, let $h(n)$~ be a multiplicative function such that for any prime $p$
\begin{equation}
\label{eq14}
h(p) = \frac{g(p)}{1-g(p)}.
\end{equation}
For $z\ge 2$ and $a\ge 1$ define $H_a = H_a(z)$ to be the sum
\begin{equation}
\label{eq15}
H_a = \sum_{\substack{k\le z\\ (k,a) = 1}}\mu^2(k)h(k).
\end{equation}
Also we need the product $P_a(z)$, defined as follows
\begin{equation}
\label{eq16}
P_a(z) = \prod_{\substack{p\le z \\ (p,a)=1}}p.
\end{equation}
Finally, consider the sequence
\begin{equation}
\rho_d = \frac{\mu(d)h(d)}{H_a g(d)}\sum_{\substack{k\le z/d\\ kd\,|\,P_a(z)}}\mu^2(k)h(k),\qquad d\ge 1.
\label{eq17}
\end{equation}
According to Selberg's sieve method (see\ \cite[ch.~3]{13}, also see\ \cite[ch.~7]{14}), this coefficients minimize the quadratic form
$$
B = \sum_{d_1,d_2\,|\,P_a(z)}\rho_{d_1}\rho_{d_2}g([d_1,d_2])
$$
of real variables $\rho_{d_1}$, $\rho_{d_2}$, for which $\rho_1 = 1$. The coefficients $\rho_d$ are defined in \eqref{eq17} satisfy the following properties: $B = 1/H_a(z)$, $|\rho_d|\le 1$, $\rho_d = 0$ for $d>z$ or $d\nmid P_a(z)$. Below we prove some lemmas about these coefficients.
\begin{lemma}
\label{l7}
Let $g(n)$ and $h(n)$ be multiplicative functions, for which conditions \eqref{eq13}, \eqref{eq14} hold, and $H_a$ be defined in~\eqref{eq15}. Further, let $p$ be a prime number, and $a\ge 1$ be an integer. Then for $(d, p)= 1$ and $(dp, a) = 1$ we have
$$
\rho_{dp} = -\rho_d +\frac{\mu(d)h(d)}{H_ag(d)}\sum_{\substack{z/(dp)<k\le z/d\\(k, dp a) = 1}}\mu^2(k)h(k).
$$
\end{lemma}

\begin{proof}
Using \eqref{eq17}, we obtain
\begin{equation}
\rho_{dp} = -\frac{h(p)}{g(p)}\, \frac{\mu(d)h(d)}{H_ag(d)}\sum_{\substack{l\le z/(dp) \\(l, dp a) = 1}} \mu^2(l) h(l) = -\frac{h(p)}{g(p)}\, \frac{\mu(d)h(d)}{H_ag(d)} S.
\label{eq18}
\end{equation}
For the sum $S$ we have
\begin{align*}
S &= \sum_{\substack{l\le z/d \\ (l, dpa) = 1}}\mu^2(l)h(l) - \sum_{\substack{z/(dp) <l\le z/d \\ (l, dpa) = 1}}\mu^2(l)h(l)
\\
&= \sum_{\substack{l\le z/d \\ (l, da) = 1}}\mu^2(l)h(l) - \sum_{\substack{l\le z/d \\ (l, da) = 1 \\ p\,|\,l}}\mu^2(l)h(l) - R,
\end{align*}
where
$$
R = \sum_{\substack{z/(dp) < l \le z/d\\(l, dp a) = 1}} \mu^2(l) h(l).
$$
Further, we have
$$
\sum_{\substack{l\le z/d \\(l, d a) = 1 \\ p\,|\,l}} \mu^2(l) h(l) = h(p)\sum_{\substack{k\le z/(dp) \\(kp,\,d a) = 1 \\ (k, p)= 1}} \mu^2(k) h(k) = h(p) \sum_{\substack{k\le z/(dp)\\(k, dp a) = 1}} \mu^2(k) h(k) = h(p) S.
$$
Hence,
$$
S = \frac{1}{1+ h(p)}\Biggl( \sum_{\substack{l\le z/d\\(l, d a) = 1}} \mu^2(l) h(l) - R\Biggr).
$$
Substituting the obtained equality for \eqref{eq18} and using that for a prime $p$ one has
$$
\frac{h(p)}{g(p)} = 1+ h(p),
$$
we conclude the proof.
\end{proof}

\begin{lemma}
\label{l8}
Suppose that functions $g(n)$ and $h(n)$ are satisfied the conditions \eqref{eq13} and \eqref{eq14} respectively, and let $H_a$ be defined in~\eqref{eq15}. Further, let $p_1, \dots, p_k$ be distinct primes and 
$$
(d, p_1\dotsb p_k) = (dp_1\dotsb p_k, a) = 1.
$$
Then
$$
\rho_{d p_1\dotsb p_k} = (-1)^k\biggl( \rho_d - \frac{\mu(d)h(d)}{H_a g(d)}R\biggr),
$$
where
$$
R = \sum_{m=1}^k\biggl( \prod_{i = 1}^{m-1}\frac{h(p_i)}{g(p_i)}\biggr) \sum_{\substack{z/(d \alpha_m)< l \le z/(d \alpha_{m-1}) \\ (l,\,d \alpha_m a) = 1}} \mu^2(l) h(l),\qquad \alpha_m = \prod_{i = 1}^{m} p_i.
$$
\end{lemma}

\begin{proof}
The proof is by induction on $k$. 
It follows from the previous lemma that this is true for $k=1$. Assume that the lemma is true for $k = r\ge 1$, and prove it for $k = r+1$. We have
$$
\rho_{dp_1\dotsb p_{r+1}} = (-1)^r \biggl( \rho_{d p_1} + \frac{\mu(d)h(d)}{H_a g(d)}\, \frac{h(p_1)}{g(p_1)}R \biggr),
$$
where
\begin{align*}
\frac{h(p_1)}{g(p_1)}R &= \frac{h(p_1)}{g(p_1)}\sum_{m=1}^r\biggl( \prod_{i = 2}^{m}\frac{h(p_i)}{g(p_i)}\biggr)\sum_{\substack{z/(d\alpha_{m+1})< l \le z/(d\alpha_{m}) \\ (l,d \alpha_{m+1} a) = 1}} \mu^2(l) h(l)
\\
&=\sum_{m=2}^{r+1}\biggl( \prod_{i = 1}^{m-1}\frac{h(p_i)}{g(p_i)}\biggr) \sum_{\substack{z/(d \alpha_{m})< l \le z/(d \alpha_{m-1}) \\ (l,\,d \alpha_{m} a) = 1}} \mu^2(l) h(l).
\end{align*}
Hence,
\begin{align*}
\rho_{dp_1\cdots p_{r+1}} &= (-1)^{r}\Biggl( -\rho_d+\frac{\mu(d)h(d)}{H_ag(d)}\sum_{\substack{z/(dp_1)<l\le z/d\\(l, dp_1 a) = 1}}\mu^2(l)h(l)
\\
&\qquad+\frac{\mu(d)h(d)}{H_ag(d)}\sum_{m=2}^{r+1}\biggl( \prod_{i = 1 }^{m-1}\frac{h(p_i)}{g(p_i)}\biggr) \sum_{\substack{z/(d \alpha_{m})< l \le z/(d \alpha_{m-1}) \\ (l,\,d \alpha_{m} a) = 1}} \mu^2(l) h(l)\Biggr).
\end{align*}
This concludes the proof.
\end{proof}

\begin{lemma}
\label{l9}
Let $P$, $M$, $q$, $a$, $\delta$~ be square-free integers satisfying the following conditions: $(q,M) = 1$; $qM\,|\,P$; $(P,a) = (P, \delta) = 1$. Let $h(n)$~ be a multiplicative function and $q = r_1r_2\cdots r_s$~ be the prime factorization of $q$. Set
\begin{align}
R_k(M,\delta) &= R_k(M,\delta,z,q,a)
\nonumber
\\
&=\sum_{\substack{d\le z \\ (d,a) = 1 \\ (d, P) = M\\ d\equiv\, 0\,(\operatorname{mod}\delta)}}\mu(d)h(d) \sum_{\substack{z/(d\alpha_k)< l \le z/(d\alpha_{k-1}) \\ (l, d\alpha_k a) = 1}} \mu^2(l) h(l),\qquad \alpha_k = \prod_{i = 1}^{k} r_i,
\label{eq19}
\end{align}
then
$$
R_k(M,\delta)\ll\mu^2(M\delta)h(M\delta)\prod_{p\,|\,P}(1+h(p)).
$$
In particular, if $P = p_1p_2 \cdots p_m$ is the product of $m$ distinct primes and $h(p)\ll 1$, then
$$
R_k(M,\delta)\ll_m \mu^2(M\delta)h(M\delta).
$$
\end{lemma}

\begin{proof}
Since $(l,d)=1$ in the sum \eqref{eq19}, then introducing the notation $n=dl$, we obtain
$$
R_k(M,\delta) = \sum_{\substack{z/\alpha_k< n \le z/\alpha_{k-1}}} \mu^2(n) h(n) \sum_{\substack{d\,|\,n \\ d\le z \\ (d,a)=1 \\ (d, P) = M\\ d\equiv 0\,(\operatorname{mod} \delta )\\ (n/d, d \alpha_k a) = 1 }}\mu(d), \qquad \alpha_k = \prod_{i = 1}^{k} r_i.
$$
Since $d\equiv 0\,(\operatorname{mod} \delta )$ we have $n\equiv 0\,(\operatorname{mod} \delta) $. Since $(d,a) = 1$ and $(n/d,a)=1$, it follows that $(n,a)=1$. Further, since $(d,P) = M$ we have $n\equiv 0\,(\operatorname{mod} M) $. Finally, note that $(d, \alpha_k) = \bigl( d, \prod_{i=1}^{k} r_i \bigr) = 1$. Assume the converse, if $r_i\,|\,(d,\alpha_k)$, then, on the one hand, $r_i\,|\,q$ and $r_i\,|\,P$; on the other hand , $r_i\,|\,(d,P) = M$. Since $(M, q) = 1$, this is a contradiction. Since $(n/d, \alpha_k) =1 $, it follows that $(n, \alpha_k) = 1$. Thus, we get
\begin{align}
R_k(M,\delta) &= \sum_{\substack{z/\alpha_k< n \le z/\alpha_{k-1} \\ n\equiv 0\,(\operatorname{mod} M\delta) \\ (n, a\alpha_k )= 1}} \mu^2(n) h(n) \sum_{\substack{d\,|\,n \\ d\le z \\ (d,a)=1 \\ (d, P) = M\\ d\equiv 0\,(\operatorname{mod} \delta )\\ (n/d, d \alpha_k a) = 1 }}\mu(d)
\nonumber
\\
&=\sum_{\substack{z/\alpha_k< n \le z/\alpha_{k-1} \\ n\equiv 0\,(\operatorname{mod} M\delta) \\ (n, a\alpha_k )= 1}} \mu^2(n) h(n) \sum_{\substack{d\,|\,n \\ (d, P) = M\\ d\equiv 0\,(\operatorname{mod} \delta )}}\mu(d)
\nonumber
\\
&= \sum_{\substack{z/\alpha_k< n \le z/\alpha_{k-1} \\ n\equiv 0\,(\operatorname{mod} M\delta) \\ (n, a\alpha_k )= 1}} \mu^2(n) h(n) W_n(M,\delta).
\label{eq20}
\end{align}
Since $M\,|\,P$ and $(\delta, P) = 1$, we have
$$
W_n(M,\delta) = \sum_{\substack{d\,|\,n \\ (d, P) = M\\ d\equiv 0\,(\operatorname{mod} \delta )}}\mu(d) = \mu(M) \sum_{\substack{d\,|\,(n/M) \\ (d, P/M) = 1\\ (d, M)=1 \\ d\equiv 0\,(\operatorname{mod} \delta ) }}\mu(d) = \mu(M)\sum_{\substack{d\,|\,n \\ (d, P) = 1 \\ d\equiv 0\,(\operatorname{mod} \delta ) }}\mu(d),
$$
where we used the fact that the condition $d\,|\,(n/M)$ is equivalent to $d\,|\,n$ and $(d,M) = 1$. Hence, we get
$$
W_n(M,\delta)\,{=}\, \mu(M)\!\sum_{\substack{d\,|\,n \\ d\equiv 0\,(\operatorname{mod} \delta )}}\mu(d)\sum_{\Delta\,|\,(d,P)}\mu(\Delta) =\mu(M) \!\sum_{\Delta\,|\,(n,P)} \mu(\Delta) \sum_{\substack{d\,|\,n \\ d\equiv 0\,(\operatorname{mod} \delta ) \\ d\equiv 0\,(\operatorname{mod} \Delta ) }}\mu(d).
$$
Further, since $\Delta\,|\,P$ and $(\delta,P) = 1$, we get $(\delta,\Delta) = 1$, whence
\begin{align*}
W_n(M,\delta) &= \mu(M) \sum_{\Delta\,|\,(n,P)}\mu(\Delta)\sum_{\substack{d\,|\,n \\ d\equiv 0\,(\operatorname{mod} \delta\Delta ) }}\mu(d)
\\
&= \mu(M) \sum_{\Delta\,|\,(n,P)}\mu(\Delta)\sum_{d\,|\,(n/(\delta\Delta))}\mu(\delta\Delta d)
\\
&= \mu(M\delta) \sum_{\Delta\,|\,(n,P)}\mu^2(\Delta)\sum_{d\,|\,(n/(\delta\Delta))}\mu(d)
\\
&=\mu(M\delta) \sum_{\Delta\,|\,(n,P)}\mu^2(\Delta)\, \mathbf{1}\biggl( \Delta = \frac{n}{\delta}\biggr) = \mu(M\delta)\, \mathbf{1}\biggl( \frac{n}{\delta}\biggm|P\biggr),
\end{align*}
where we use the notation
$$
\mathbf{1}(A) = \begin{cases}
1, &\text{if condition }A\text{ holds},
\\
0 &\text{otherwise}.
\end{cases}
$$
Substituting the resulting expression of $W_n(M,\delta)$ for~\eqref{eq20}, we obtain
\begin{align*}
R_k(M,\delta) &= \mu(M\delta) \sum_{\substack{z/\alpha_k< n \le z/\alpha_{k-1} \\ n\equiv 0\,(\operatorname{mod} \delta M) \\ (n, a \alpha_k )= 1 \\ n/(\delta\,|\,P)}} \mu^2(n) h(n)
\\
&=\mu(M\delta) h(M\delta) \sum_{\substack{z/(\delta M\alpha_k)< n \le z/(\delta M\alpha_{k-1}) \\ (n, a \alpha_k \delta M)= 1 \\ n\,|\,(P/M)}} \mu^2(n) h(n)
\\
&\ll\mu^2(M\delta) h(M\delta) \sum_{n\,|\,P}\mu^2(n)h(n) = \mu^2(M\delta) h(M\delta) \prod_{p\,|\,P}(1+h(p)),
\end{align*}
giving the lemma.
\end{proof}

Now we prove the main lemma.

\begin{lemma}
\label{l10}
Let $m\ge 1$~ be an arbitrary fixed integer. Further, let $\{f\}_{i=1}^m$ be a set of functions such that for any prime $f_i(p)\ll_m (\ln p)^{m +1}/p$. Put
$$
T_m(z;a) = \sum_{\substack{p_1\le z\\p_1\,\nmid\,a}}f_1(p_1) \dots \sum_{\substack{p_m\le z \\p_m\,\nmid\, a}}f_m(p_m) \sum_{\substack{d_1, d_2 \,|\, P_a(z) \\ d_1, d_2\le z \\ [d_1, d_2]\equiv 0\,(\operatorname{mod} [p_1, \dots, p_m])}} \frac{\rho_{d_1}\rho_{d_2}}{\varphi([d_1,d_2])}J_{[d_1,d_2]}(1),
$$
where the function $J_d(s)$ is defined in~\eqref{eq6}, the coefficients $\rho_d$ are defined in~\eqref{eq17}, and the product $P_a(z)$ is defined in~\eqref{eq16}. Then the sum $T_m(z;a)$ satisfies the estimate
$$
T_m(z;a)\ll_m \frac{1}{H_a(z)},
$$
where $H_a(z)$ is defined in~\eqref{eq15}.
\end{lemma}

\begin{proof}
Denote $g(d) = J_d(1)/\varphi(d)$, then
$$
g(p) = \frac{1}{p(p-1)\ln(1/(1-1/p))}
$$
and for any $p\ge 3$ we have
\begin{equation}
\label{eq21}
\frac{1}{p}\le g(p) \le \frac{1}{p-1}.
\end{equation}
Hence, $g(p)\le 1/2$ for $p\ge 3$. Then since $g(2) = 1/(2\ln{2})$, it follows that for any prime
\begin{equation}
\label{eq22}
0<g(p)\le \frac{1}{2\ln{2}}<1.
\end{equation}
Further, from the definition of $g(d)$ we have
$$
T_m(z;a) = \sum_{\substack{p_1\le z\\p_1\,\nmid\,a}}f_1(p_1) \dots \sum_{\substack{p_m\le z \\p_m\,\nmid\, a}}f_m(p_m) \sum_{\substack{d_1, d_2 \,|\, P_a(z) \\ d_1, d_2\le z \\ [d_1, d_2]\equiv 0\,(\operatorname{mod} [p_1, \dots, p_m])}}{\rho_{d_1}\rho_{d_2}}g([d_1, d_2]).
$$
Using the estimate \eqref{eq21} for $p\ge 3$, we find that
$$
\frac{1}{p-1}\le h(p)\le\frac{1}{p-2}.
$$
According to \eqref{eq14} we get
$$
g(p) = \frac{h(p)}{h(p)+1}.
$$
Hence, for the square-free $d$ we have
$$
\frac{1}{g(d)} = \prod_{p\,|\,d}\biggl(1+\frac{1}{h(p)}\biggr) = \sum_{\delta \,|\, d}\frac{1}{h(\delta)}.
$$
Since $g(n)$ is multiplicative and $d_1$ and $d_2$ are square-free, we have
$$
g([d_1, d_2]) = \frac{g(d_1)g(d_2)}{g((d_1, d_2))} = g(d_1)g(d_2)\sum_{\delta \,|\, (d_1 , d_2)}\frac{1}{h(\delta )}.
$$
Then, 
\begin{align*}
&T_m(z;a)
\\
&= \sum_{\substack{p_1\le z\\ p_1\,\nmid\,a}}f_1(p_1) \dots \sum_{\substack{p_m\le z \\ p_m\,\nmid\, a}}f_m(p_m)\sum_{\substack{\delta\,|\, P_a(z) \\ \delta\le z}}\frac{1}{h(\delta)}\sum_{\substack{d_1, d_2 \,|\, P_a(z) \\ d_1, d_2\le z\\ [d_1, d_2]\equiv 0\,(\operatorname{mod} [p_1, \dots, p_m]) \\ d_1, d_2 \equiv 0\,(\operatorname{mod} \delta)}}\rho_{d_1}g(d_1)\rho_{d_2}g(d_2).
\end{align*}
Without loss of generality, we may assume that all $p_1, p_2, \dots, p_m$ are pairwise distinct. Indeed, if $p_1 = p_2$, then
$$
f_1(p_1)f_2(p_1)\ll\frac{(\ln p_1)^{2m+2}}{p_1^2}\ll_m\frac{\ln p_1}{p_1},\qquad [p_1,p_2,\dots,p_m] = [p_1, p_3, \dots, p_m],
$$
and the estimation of the sum $T_m(z;a)$ reduces to the estimation of the sum of the same type, but with a~smaller value of the parameter $m$. Thus,
\begin{equation}
T_m(z;a) = \sum_{\substack{p_1\le z\\ p_1\,\nmid\,a} }f_1(p_1) \sum_{\substack{p_2\le z\\ p_2\,\nmid\,a \\ p_2\neq p_1} }f_2(p_2) \dots \sum_{\substack{p_m\le z\\ p_m\,\nmid\,a \\ p_m\notin \{p_1, p_2, \dots, p_{m-1}\}} }f_m(p_m)\times S,
\label{eq23}
\end{equation}
where
$$
S = \sum_{\substack{\delta\,|\, P_a(z) \\\delta\le z}}\frac{1}{h(\delta)}\sum_{\substack{d_1, d_2 \,|\,P_a(z) \\ d_1, d_2\le z\\ [d_1, d_2]\equiv\, 0\,(\operatorname{mod} p_1 \dotsb p_m) \\ d_1, d_2 \equiv\, 0\,(\operatorname{mod} \delta)}}\rho_{d_1}g(d_1)\rho_{d_2}g(d_2).
$$
Let $P_m = p_1\cdots p_m$. Note that the condition
$$
[d_1 , d_2] \equiv 0\,(\operatorname{mod} P_m)
$$
is equivalent to
$$
([d_1 , d_2], P_m) = [(d_1, P_m), (d_2, P_m) ] = P_m.
$$
Let $A = (d_1, P_m)$ and $B = (d_2, P_m)$. Then,
$$
\sum_{\substack{d_1, d_2 \,|\, P_a(z) \\ d_1, d_2\le z\\ [d_1, d_2]\equiv 0\,(\operatorname{mod} P_m) \\ d_1, d_2 \equiv 0\,(\operatorname{mod} \delta)}}\rho_{d_1}g(d_1)\rho_{d_2}g(d_2) = \sum_{A\,|\,P_m}\sum_{\substack{B\,|\,P_m \\ [A,B] = P_m}} S_{\delta}(A)S_{\delta}(B),
$$
where for an integer $N$ we put
$$
S_{\delta}(N) = \sum_{\substack{d\le z \\ d\,|\,P_a(z) \\ (d, P_m) = N \\ d\equiv 0\,(\operatorname{mod} \delta)}}\rho_d g(d).
$$
Further, let $D=(A, B)$, then
\begin{align*}
S &= \sum_{\substack{\delta\,|\, P_a(z) \\\delta\le z}}\frac{1}{h(\delta)}\sum_{A\,|\,P_m}\sum_{\substack{B\,|\,P_m \\ [A,B] = P_m}} S_{\delta}(A)S_{\delta}(B)
\\
&= \sum_{D \,|\, P_m} \sum_{\substack{A\,|\, P_m\\ A \equiv  0\,(\operatorname{mod} D)}} \sum_{\substack{\delta\,|\, P_a(z)\\ \delta\le z}}\frac{1}{h(\delta)} S_\delta(A) S_{\delta}(B),
\end{align*}
where the values $A$, $B$ and $D$ are satisfy the relation $A B = {D P_m}$.

Note that $(\delta,P_m/D)=1$. Indeed, if there exists a prime $p$ such that $p\,|\, (\delta,P_m/D)$, then $p\,|\, \delta$, $p\,|\, P_m$ and $(p, D)=1$. It follows from this that $(p,A)\,{=}\,1$ 
or $(p, B)=1$. Without loss of generality, assume that $(p,A)=1$. Then all $d$ corresponding to the summands of the sum $S_\delta(A)$ are divisible by $p$. Since $p\,|\,P_m$, it follows from the condition $(d, P_m) = A$ in the sum $S_\delta(A)$ that $p\,|\,A$. This is a contradiction. Thus,
\begin{align}
S &= \sum_{D \,|\, P_m} \sum_{\substack{A\,|\, P_m\\ A \equiv 0\,(\operatorname{mod} D)}} \sum_{\substack{\delta\,|\, P_a(z) \\ \delta\le z \\ (\delta,P_m/D) = 1}}\frac{1}{h(\delta)} S_\delta(A) S_{\delta}(B)
\nonumber
\\
&=\sum_{D \,|\, P_m} \sum_{\substack{A\,|\, P_m\\ A \equiv\, 0\,(\operatorname{mod} D)}}\sum_{q\,|\,D} \sum_{\substack{\delta\,|\, P_a(z) \\ \delta\le z \\ (\delta,P_m/D) = 1 \\ (\delta, D) = q}}\frac{1}{h(\delta)} S_\delta(A) S_{\delta}(B)
\nonumber
\\
&=\sum_{D \,|\, P_m} \sum_{\substack{A\,|\, P_m\\ A \equiv\, 0\,(\operatorname{mod} D)}}\sum_{q\,|\,D}\frac{1}{h(q)} \sum_{\substack{\delta\,|\, P_a(z) \\ \delta\le z/q \\ (\delta,P_m) = 1 }}\frac{1}{h(\delta)} S_{\delta q}(A) S_{\delta q}(B).
\label{eq24}
\end{align}
Let us express the value $S_{\delta q}(A)$ in terms of the value $S_{\delta}(1)$. We have
$$
S_{\delta q}(A) = \sum_{\substack{d\le z \\ d\,|\,P_a(z) \\ (d, P_m) = A\\ d\equiv 0\,(\operatorname{mod} \delta q )}}\rho_d g(d) =
\sum_{\substack{d\le z/q \\ d\,|\,(P_a(z)/q) \\ (d, P_m/q) = A/q \\d\equiv 0\,(\operatorname{mod} \delta )}}\rho_{dq} g(dq).
$$
Since the condition $d\,|\,(P_a(z)/q)$ is equivalent to the condition
$$
d\,|\,P_a(z),\qquad (d,q)=1,
$$
 and $\rho_d=0$ for $d>z$,
 we get
$$
S_{\delta q}(A) = g(q)\sum_{\substack{d\le z \\ d\,|\,P_a(z) \\ (d, P_m/q) = A/q\\ (d,q) = 1 \\ d\equiv 0\,(\operatorname{mod} \delta)}}\rho_{d q} g(d) = g(q)\sum_{\substack{d\le z \\ d\,|\,P_a(z) \\ (d, P_m) = A/q\\ (d,q) = 1 \\ d\equiv 0\,(\operatorname{mod} \delta)}}\rho_{d q} g(d).
$$
Since for $(d,q) > 1$ we have $\rho_{dq} = 0$, it follows that the condition $(d,q) = 1$ in the above sum can be omitted. Thus, we get
$$
S_{\delta q}(A) = g(q)\sum_{\substack{d\le z \\ d\,|\,P_a(z) \\ (d, P_m) = A/q \\ d\equiv 0\,(\operatorname{mod} \delta)}}\rho_{d q} g(d).
$$
Further, since $(\delta, P_m) = 1$ and $Aq^{-1}\,|\,P_m$, we conclude that $(\delta,{A}{q}^{-1 }) = 1$. Hence,
\begin{equation}
\label{eq25}
S_{\delta q}(A) = g(q)\sum_{\substack{d\le z/(Aq^{-1}) \\ d\,|\,(P_a(z)/(Aq^{-1})) \\ (d, P_m/(Aq^{-1})) = 1 \\ d\equiv 0\,(\operatorname{mod} \delta)}}\rho_{d A} g(Aq^{-1}d) = g(A)\sum_{\substack{d\le z \\ d\,|\,P_a(z) \\ (d, P_m) = 1 \\ d\equiv 0\,(\operatorname{mod} \delta)}}\rho_{d A} g(d).
\end{equation}
Let $A>1$ and $A = r_1 r_2 \cdots r_s$ be the factorization of $A$. Then, $(-1)^s = \mu(A)$. Since $(d,A) = (dA,a) = 1$, Lemma \ref{l8} implies that
$$
\rho_{dA} = \mu(A)\biggl(\rho_d-\frac{\mu(d)h(d)}{H_a(z)g(d)}R\biggr),
$$
where
$$
R = \sum_{k=1}^{\omega(A)}\biggl( \prod_{i = 1}^{k-1} \frac{h(r_i)}{g(r_i)}\biggr) \sum_{\substack{z/(d\alpha_k)< l \le z/(d\alpha_{k-1}) \\ (l,d \alpha_k a) = 1}} \mu^2(l) h(l),\qquad \alpha_k = \prod_{i = 1}^{k} r_i.
$$
From \eqref{eq25} we get
$$
S_{\delta q}(A) = \mu(A)g(A)\biggl(S_{\delta}(1)-\frac{R'}{H_a(z)}\biggr),
$$
where
$$
R' = \sum_{k=1}^{\omega(A)}\biggl( \prod_{i = 1}^{k-1}\frac{h(r_i)}{g(r_i)}\biggr) \sum_{\substack{d\le z \\ d\,|\,P_a(z) \\ (d, P_m) = 1 \\ d\equiv 0\,(\operatorname{mod} \delta)}}\mu(d)h(d) \sum_{\substack{z/(d\alpha_k)< l \le z/(d\alpha_{k-1}) \\ (l,d \alpha_k a) = 1}} \mu^2(l) h(l).
$$
From Lemma \ref{l9} and the equality $h(p)/g(p)=1+h(p)$ we obtain the following estimate of $R'$:
$$
R'=\sum_{k=1}^{\omega(A)}\biggl( \prod_{i = 1}^{k-1}(1+h(r_i))\biggr) R_k(1,\delta)\ll\omega(A)\prod_{p\,|\,P_m}(1+h(p))^2\mu^2(\delta) h(\delta)\ll_m \mu^2(\delta) h(\delta).
$$
Thus, for $A>1$ we have
\begin{equation}
S_{\delta q}(A) = \mu(A)g(A)\biggl( S_{\delta }(1)+O_m\biggl( \frac{\mu^2(\delta) h(\delta)}{H_a}\biggr) \biggr).
\label{eq26}
\end{equation}
From \eqref{eq25} for $A = 1$ we get $S_{\delta q} (1) = S_{\delta}(1)$, thus the equality \eqref{eq26} also holds for $A = 1$.

Then from \eqref{eq24}, we obtain
\begin{align*}
S &= \sum_{D \,|\, P_m} \sum_{\substack{A\,|\, P_m\\ A \equiv 0\,(\operatorname{mod} D)}}\sum_{q\,|\,D}\frac{1}{h(q)} \sum_{\substack{\delta\,|\, P_a(z) \\ \delta\le z/q \\ (\delta,{P_m})=1}}\frac{1}{h(\delta)} \mu(A)g(A)\mu(B)g(B)
\\
&\qquad\times\biggl( S_{\delta}(1)+O_m\biggl(\frac{\mu^2(\delta) h(\delta)}{H_a(z)} \biggr) \biggr)^2
\\
&=\mu(P_m) g(P_m) \sum_{D \,|\, P_m}\mu(D) g(D) \sum_{\substack{A\,|\, P_m\\ A \equiv\, 0\,(\operatorname{mod} D)}}\sum_{q\,|\,D}\frac{1}{h(q)}
\\
&\qquad \times \sum_{\substack{\delta\,|\, P_a(z) \\ \delta\le z/q \\ (\delta,{P_m}) = 1 }}\frac{1}{h(\delta)} \biggl( S_{\delta}(1)+O_m\biggl(\frac{\mu^2(\delta) h(\delta)}{H_a(z)} \biggr) \biggr)^2.
\end{align*}
Since for any square-free $d$ we have
$$
\frac{g(d)}{h(d)}\le 1,
$$
it follows that
\begin{align}
S &\ll_m g(P_m)\sum_{D\,|\,P_m}\frac{g(D)}{h(D)}\sum_{\substack{A\,|\,P_m \\ A\equiv 0\, (\operatorname{mod} D)}}\sum_{q\,|\,D}h\biggl(\frac{D}{q}\biggr)
\nonumber
\\
&\qquad\times \sum_{\substack{\delta\,|\, P_a(z) \\ \delta\le z/q \\ (\delta,{P_m}) = 1 }}\frac{1}{h(\delta)} \biggl( S_{\delta}(1)+\frac{\mu^2(\delta) h(\delta)}{H_a(z)} \biggr)^2
\nonumber
\\
&\ll_m g(P_m) W \sum_{\substack{\delta\,|\, P_a(z) \\ \delta\le z \\ (\delta,{P_m}) = 1 }}\frac{1}{h(\delta)} \biggl( S_{\delta}^2(1)+\frac{\mu^2(\delta) h^2(\delta)}{H^2_a(z)} \biggr),
\label{eq27}
\end{align}
where the sum $W$ is defined by the equality
$$
W = \sum_{D\,|\,P_m}\sum_{\substack{A\,|\,P_m \\ A\equiv 0\,(\operatorname{mod} D)}}\sum_{d\,|\,D}h(d).
$$
Note that $W\ll 4^m$. Indeed,
\begin{align*}
W &= \sum_{D\,|\,P_m}\prod_{p\,|\,D}(1+h(p))\sum_{\substack{A\,|\,P_m \\ A\equiv 0\, (\operatorname{mod} D)}}1 = \sum_{D\,|\,P_m} \prod_{p\,|\,D}(1+h(p)) \frac{\tau(P_m)}{\tau(D)}
\\
&= \tau(P_m)\sum_{D\,|\,P_m}\prod_{p\,|\,D}\frac{1+h(p)}{2} = \tau(P_m)\prod_{p\,|\,P_m}\frac{3+h(p)}{2} = \prod_{p\,|\,P_m}(3+h(p))\ll 4^m.
\end{align*}
Then, from \eqref{eq27} and \eqref{eq15} we find that
\begin{equation}
\label{eq28}
S\ll_m g(P_m)\Biggl( \sum_{\substack{\delta\,|\, P_a(z) \\ \delta\le z \\ (\delta,{P_m}) = 1 }}\frac{S_{\delta}^2(1)}{h(\delta)}+\frac{1}{H_a(z)}\Biggr).
\end{equation}
Note that
\begin{align*}
S_{\delta}(1) &= \sum_{\substack{d\le z \\ d\,|\,P_a(z) \\ (d, P_m) = 1\\ d\equiv 0\,(\operatorname{mod} \delta)}}\rho_d g(d) = \sum_{\substack{d\le z \\ d\,|\,P_a(z) \\ d\equiv 0\,(\operatorname{mod} \delta )}}\rho_d g(d)- \sum_{\substack{\Delta\,|\,P_m\\ \Delta>1}} S_{\delta}(\Delta)
\\
&=x_{\delta}-\sum_{\substack{\Delta\,|\,P_m\\ \Delta>1}} \mu(\Delta)g(\Delta)\biggl( S_{\delta}(1)+ O_m\biggl( \frac{\mu^2(\delta)h(\delta)}{H_a(z)}\biggr) \biggr)
\\
&=x_{\delta}-S_{\delta}(1)\sum_{\substack{\Delta\,|\,P_m\\ \Delta>1}} \mu(\Delta)g(\Delta) + O_m\Biggl( \frac{\mu^2(\delta)h(\delta)}{H_a(z)}\sum_{\substack{\Delta\,|\,P_m \\ \Delta > 1}}\mu^2(\Delta)g(\Delta)\Biggr).
\end{align*}
Then, we obtain
$$
S_{\delta}(1) \sum_{\Delta\,|\, P_m}\mu(\Delta)g(\Delta) = x_{\delta} + O_{m}\biggl(\frac{\mu ^2(\delta)h(\delta)}{H_a(z)} \biggr).
$$
Since $0<g(p) \le 1/(2\ln 2)$ for any prime $p$, it follows that
$$
\sum_{\delta\,|\, P_m}\mu(\Delta)g(\Delta) = \prod_{p\,|\,P_m}(1-g(p)) \gg_m 1.
$$
Hence, we get
$$
S^2_{\delta}(1)\ll_m x^2_{\delta}+\frac{\mu^2(\delta) h^2(\delta)}{H_a^2(z)}.
$$

Note that
$$
\sum_{\substack{\delta\,|\, P_a(z) \\ \delta\le z}}\frac{1}{h(\delta)} x^2_{\delta} = \frac{1}{H_a(z)}.
$$
Indeed,
\begin{align*}
\frac{1}{H_a(z)} &= \sum_{d_1,d_2\,|\,P_a(z)} \rho_{d_1}\rho_{d_2}g(d_1)g(d_1)\frac{1}{g((d_1,d_2))}
\\
&=\sum_{\substack{\delta\,|\,P_a(z)\\\delta \le z}}\frac{1}{h(\delta)}\Biggl( \sum_{\substack{d\le z \\ d\,|\,P_a(z) \\ d\equiv 0\,(\operatorname{mod} \delta )}}\rho_d g(d)\Biggr)^2 = \sum_{\substack{\delta\,|\, P_a(z) \\ \delta\le z}}\frac{1}{h(\delta)} x^2_{\delta}.
\end{align*}
Thus, from \eqref{eq28} we get
$$
S\ll_m \frac{g(P_m)}{H_a(z)} = \frac{g(p_1)g(p_2)\cdots g(p_m)}{H_a(z)}\ll_m \frac{1}{p_1p_2\cdots p_m H_a(z)}.
$$
Substituting this estimate for \eqref{eq23}, we obtain that
$$
T_m(z;a) \ll \sum_{p_1\le z}\frac{f_1(p_1)}{p_1} \sum_{p_2\le z}\frac{f_1(p_2)}{p_2} \cdots \sum_{p_m\le z}\frac{f_m(p_m)}{p_m}\, |S|\ll_m \frac{1}{H_a(z)}.
$$
The claim follows.
\end{proof}

\begin{lemma}
\label{l11}
Let $g(n)$ be a multiplicative function, for which

--~for some constant $A_1\ge 1$ and any prime $p$
\begin{equation}
	\label{eq29}
	0\le g(p)\le 1-\frac{1}{A_1};
\end{equation}

--~for some values $\kappa$, $L$, $A_2>0$ the following inequalities hold
\begin{equation}
	\label{eq30}
	-L\le \sum_{u\le p<v}g(p)\ln p-\kappa\ln\biggl(\frac{v}{u}\biggr)\le A_2.
\end{equation}

Then for $W(z) = \prod_{p<z}(1-g(p))$ we have
$$
W(z) = \prod_{p}(1-g(p))\biggl(1-\frac{1}{p}\biggr)^{-\kappa}\frac{e^{-\gamma\kappa}}{(\ln z)^{\kappa}}\biggl(1+O\biggl(\frac{L}{\ln z}\biggr)\biggr),
$$
where $\gamma$~ is the Euler--Mascheroni constant.
\end{lemma}

\begin{proof}[{\rm see\ in~\cite[\S\,5, Lemma~5.3]{13}}]
\end{proof}

\begin{lemma}
\label{l12}
Let $g(n)$~ and $h(n)$~ be multiplicative functions, and
$$
0\le g(p)<1,\qquad h(p) = \frac{g(p)}{1-g(p)}
$$
for any prime $p$. Further, suppose that the conditions \eqref{eq29} and \eqref{eq30} are hold. Put
\begin{equation}
	\label{eq31}
	H(z) = \sum_{k\le z}\mu^2(k)h(k), \qquad W(z) = \prod_{p<z}(1-g(p)).
\end{equation}

Then, we have
$$
\frac{1}{H(z)} = W(z)e^{\gamma\kappa}\Gamma(\kappa+1)\biggl(1+O\biggl(\frac{L}{\ln z }\biggr)\biggr),
$$
where $\gamma$~ is the Euler--Mascheroni constant.
\end{lemma}

\begin{proof}[{\rm see\ in~\cite[\S\,5, Lemma~5.4]{13}}]
\end{proof}

From Lemmas \ref{l11} and \ref{l12} for $H(z)$ defined in~\eqref{eq31}, we get
\begin{equation}
\label{eq32}
\frac{1}{H(z)} = \prod_{p}(1-g(p)) \biggl(1-\frac{1}{p}\biggr)^{-\kappa} \frac{\Gamma(\kappa+1)}{(\ln z)^{\kappa}}\biggl(1+O\biggl(\frac{L}{\ln z}\biggr)\biggr).
\end{equation}
We need the main lemma of \cite{9}.

\begin{lemma}
\label{l13}
Let $d\ge 1$~ be a fixed integer, $f(q)$~ be a positive function such that
$$\sum_{q\le x}f(q)\ll x(\ln x)^{\kappa}$$ for some constant $\kappa > 0$. Then for any fixed $B>0$ there exists $A(B) > 0$ such that the inequality
$$
R = \sum_{q\le Q}\frac{f(q)}{\varphi(q)}\sum_{\substack{\chi\,\operatorname{mod} q \\ \chi\neq\chi_0}}\,\Biggl|\sum_{\substack{n\le N \\ (n, d) = 1}} \frac{\chi(n)}{\tau(n)}\Biggr| \ll x(\ln x)^{-B}
$$
holds for any $Q$, $N$ such that $Q\le \sqrt{x}\,(\ln x)^{-A}$, $N\le x$, and an implied constant is inefficient and depends only on $B$, $d$ and $f$.
\end{lemma}

\begin{proof}[{\rm see\ \cite[ Lemma~13]{9}}]
\end{proof}

\section{Proof of the theorem}
\label{s2}

We have
$$
\Phi_a(x)=\sum _{p \le x} \frac{1}{\tau (p+a)}\le \sum_{\substack{ n+a\le x \\ (n,P_a(z))=1 }} \frac{1}{\tau (n+a)}+r_1,
$$
where $r_1=r_1(z;a)$~ is the number of those $n\le z$ for which $n+a$ is prime, so that
$$
r_1\ll\frac{z}{\ln z}+a.
$$
Further, let $\rho_d$ be chosen as in~\eqref{eq17}, then $\rho_1 = 1$, $\rho_d = 0$ for $d>z$ or $d\nmid P_a(z)$; and $|\rho_d|\le 1$. Then, we get
\begin{align*}
\Phi_a(x) &\le\sum_{\substack{ n+a\le x \\ (n,P_a(z))=1 }} \frac{1}{\tau (n+a)}+r_1 \le \sum_{\substack{ n+a\le x}} \frac{1}{\tau (n+a)}\biggl(\sum_{d\,|\,(n,P_a(z))}\rho_d\biggr)^2+ r_1
\\
&=\sum_{\substack{d\,|\,P_a(z) \\ d\le z^2}}{\lambda_d}\sum_{\substack{ n+a\le x \\ n\equiv0\,(\operatorname{mod} d)}}\frac{1}{\tau(n+a)}+r_1,
\end{align*}
where $\lambda_d = \sum_{[d_1, d_2] = d}\rho_{d_1}\rho_{d_2}$. Note that
$$
|\lambda_d|\le \sum_{[d_1, d_2] = d} 1 = 3^{\omega(d)}.
$$
Since $(d,a) = 1$, we have
$$
\sum_{\substack{ n+a\le x \\ n\equiv0\, (\operatorname{mod} d)}}\frac{1}{\tau(n+a)} = \sum_{\substack{ k\le x \\ k\equiv a\,(\operatorname{mod} d)}}\frac{1}{\tau(k)} = \frac{1}{\varphi(d)}\sum_{\chi\, \operatorname{mod} d}\overline{\chi}(a)\sum_{k\le x}\frac{\chi(k)}{\tau(k)},
$$
where the summation is over all Dirichlet characters modulo $d$. Separating the contribution of the main character and denoting by $r_2$ the contribution from the remaining characters, we obtain
$$
\Phi_a(x)\le \sum_{\substack{d\,|\,P_a(z) \\ d\le z^2}}\frac{\lambda_d}{\varphi(d)} \sum_{\substack{ k\le x \\ (k, d) = 1}}\frac{1}{\tau(k)}+r_2+r_1,
$$
where
$$
r_2 = r_2(x,z;a) \ll\sum_{d\le z^2}\frac{|\lambda_d|}{\varphi(d)} \sum_{\substack{\chi\,\operatorname{mod} d \\ \chi\neq\chi_0}}\, \biggl|\sum_{\substack{ k\le x}}\frac{\chi(k)}{\tau(k)}\biggr| \ll \sum_{d\le z^2}\frac{3^{\omega(d)}}{\varphi(d)} \sum_{\substack{\chi\,\operatorname{mod} d \\ \chi\neq\chi_0}}\, \biggl|\sum_{\substack{ k\le x}}\frac{\chi(k)}{\tau(k)}\biggr|.
$$
Consider an integer $m_0\ge 1$. Then from Lemma \ref{l4} we get
\begin{align}
\Phi_a(x) &\le \sum_{\substack{d\,|\,P_a(z) \\ d\le z^2}}\frac{\lambda_d}{\varphi(d)}\biggl(\frac{x}{\sqrt{\pi\ln x}}\sum_{k=0}^{m_0}\frac{(-1)^k\binom{2k}{k}}{4^k}\, \frac{G_d^{(k)}(1)}{(\ln x)^k}+R_{m_0}(x;d)\biggr)+r_2+r_1
\nonumber
\\
&= \frac{x}{\sqrt{\pi\ln x}}\sum_{k=0}^{m_0}\frac{(-1)^k\binom{2k}{k}}{4^k(\ln x)^k} \sum_{\substack{d\,|\,P_a(z) \\ d\le z^2}}\frac{\lambda_d G_d^{(k)}(1)}{\varphi(d)}+r_3+r_2+r_1,
\label{eq33}
\end{align}
where
\begin{align*}
r_3 &\ll_{m_0}\frac{x}{(\ln x)^{m_0 + 3/2}}\sum_{d\le z^2}\frac{|\lambda_d|}{\varphi(d)}(\omega(d)+2)^{10}
\\
&\ll \frac{x}{(\ln x)^{m_0 + 3/2}}\sum_{d\le z^2} \frac{(3+1/2)^{\omega(d)}}{\varphi(d)}
\ll\frac{x(\ln z)^{3+1/2}}{(\ln x)^{m_0+3/2}}.
\end{align*}
Substituting the equality
$$
G^{(k)}_d(s) = \sum_{l = 0}^k\binom{k}{l}H^{(k-l)}(s)J_d^{(l)}(s)
$$
for~\eqref{eq33} and changing the order of summation over the variables $d$ and $l$, we obtain
\begin{equation}
\label{eq34}
\Phi_a(x)\le \frac{x}{\sqrt{\pi\ln x}}\sum_{k = 0}^{m_0}\frac{a_k}{(\ln x)^k}\sum_{l = 0}^k b_l(k)S_l(z;a)+r_3+r_2+r_1,
\end{equation}
where
\begin{gather*}
a_k = \frac{(-1)^k\binom{2k}{k}}{4^k},\qquad b_l(k) = \binom{k}{l} H^{(k-l)}(1),
\\
S_l(z;a) = \sum_{\substack{d\,|\,P_a(z) \\ d\le z^2}}\frac{\lambda_d J_d^{(l)}(1)}{\varphi(d)}.
\end{gather*}
Let us estimate the values $S_l(z;a)$. We set $g(d) = J_d(1)/\varphi(d)$, then for $l = 0$ we get
\begin{equation}
\label{eq35}
S_0(z;a) = \sum_{\substack{d_1,d_2 \,|\,P_a(z) \\ d_1, d_2\le z}}\rho_{d_1}\rho_{d_2}g([d_1,d_2]) = \frac{1}{H_a(z)}.
\end{equation}
Let $l\ge 1$, then from Lemma \ref{l6} we get
\begin{equation}
\label{eq36}
S_l(z;a) = \sum_{\substack{d\,|\,P_a(z) \\ d\le z^2}}\frac{\lambda_d}{\varphi(d)}J_d(1)Q_l.
\end{equation}
Here $Q_l$~ is a polynomial of degree $l$ in variables $I(1), I'(1), \dots, I^{(l-1)}(1)$ with integer coefficients:
$$
Q_l = \sum_{m=0}^l\sum_{\substack{ |\mathbf{i}| = m }}a_{\mathbf{i}}\prod_{\substack{r=1}}^l\bigl( I^{(r-1)}(1)\bigr)^{i_r} ,
$$
where $\mathbf{i} = (i_1,\dots, i_l)$ is an integer vector, $i_r\geqslant 0$, $|\mathbf{i}| = i_1+\dots+i_l$, and $I(s)$~ is defined in~\eqref{eq12}. For $1\le r\le l$ we have
$$
I^{(r-1)}(1) = \sum_{p\,|\,d}\mathfrak{f}_r(p),
$$
where $\mathfrak{f}_r(p)=f^{(r-1)}(1;p)$. Lemma \ref{l5} implies that for any prime $p$
$$
\mathfrak{f}_{r}(p)\ll_r\frac{(\ln p)^{r+1}}{p}\ll_{m_0}\frac{(\ln p)^{m_0+1}}{p}.
$$
For convenience, for each $r\ge 1$ we consider $i_r$ identical functions with different indices
$$
f_{r1} = f_{r_2} = \dots = f_{ri_{r}} = \mathfrak{f}_r.
$$
Hence, using this notation, we obtain
$$
(I^{(r-1)}(1))^{i_r} \,{=} \sum_{p_{r1}\,|\,d}f_{r1}(p_{r1})\cdots \sum_{p_{ri_{r}}\,|\,d}f_{ri_r}(p_{ri_r})\,{=}
\sum_{p_{r1},\,\dots,\, p_{ri_r}\,|\,d}{f}_{r1}(p_{r1})\cdots {f}_{ri_r}(p_{ri_r}).
$$
Let $j_1, j_2, \dots, j_R$~be the sequence of indices of all nonzero coordinates of the vector~$\mathbf{i}$. Then $j_1<j_2<\dots<j_{R}$, $R\le l$, and $i_r \neq 0$ if and only if $r\in \{j_1, j_2, \dots, j_R\} $. Then, we have
\begin{align*}
\prod_{\substack{r=1}}^l\bigl( I^{(r-1)}(1)\bigr)^{i_r} &= \prod_{\substack{r=1 \\ i_r \neq 0}}^l \sum_{p_{r1},\dots,p_{ri_r}\,|\,d}f_{r1}(p_{r1})\cdots f_{ri_r}(p_{ri_r})
\\
&=\sum_{p_{j_11}, \dots, p_{j_Ri_{j_R}}\,|\,d} f_{j_11}(p_{j_11})\cdots f_{j_Ri_{j_R}}(p_{j_Ri_{j_R}}).
\end{align*}
Let us associate each pair $(j_\nu,t)$, where $1\le t\le i_{j_\nu}$, with the number
$$
i = i_{j_1}+\dots+i_{j_{\nu-1}}+t.
$$
With this numbering, $f_{j_\nu t}(p_{j_\nu t})$ will be rewritten as $f_i(p_i)$, where $1\le i\le m$. Then, we get
$$
\prod_{\substack{r=1}}^l\bigl( I^{(r-1)}(1)\bigr)^{i_r} = \sum_{p_1,p_2,\dots,p_m \,|\, d}f_1(p_1)f_2(p_2)\cdots f_m(p_m).
$$
Thus, for $Q_l$ we have
\begin{equation}
\label{eq37}
Q_l = \sum_{m=0}^l \sum_{\substack{|\mathbf{i}| = m}}a_{\mathbf{i}}\,\sum_{p_1,p_2,\dots,p_m \,|\, d}f_1(p_1)f_2(p_2)\cdots f_m(p_m).
\end{equation}
Substituting \eqref{eq37} for~\eqref{eq36} and changing the order of summation, we obtain
\begin{align*}
S_l(z;a) &= \sum_{m=0}^l \sum_{\substack{|\mathbf{i}|=m}} a_{\mathbf{i}}\, \sum_{\substack{p_1 \le z \\ p_1\nmid\, a}}f_1(p_1)\cdots\sum_{\substack{p_m \le z \\ p_m\nmid\, a}}f_m(p_m)
\\
&\qquad\times\sum_{\substack{d_1,d_2\,|\,P_a(z) \\ d_1, d_2\le z \\ [d_1,d_2]\equiv 0\,(\operatorname{mod} P)}}\frac{\rho_{d_1}\rho_{d_2}}{\varphi([d_1,d_2])}J_{[d_1,d_2]}(1),
\end{align*}
where $P$~ is the least common multiple of the primes $p_{1}, p_2, \dots, p_m$. Applying \eqref{eq35} for $m=0$ and Lemma \ref{l10} for $m\ge 1$, we obtain
\begin{equation}
\label{eq38}
S_l(z;a)\ll_{m_0}\frac{1}{H_a(z)}.
\end{equation}
Now we separate the term corresponding to $k = 0$ in  sum \eqref{eq34}, and apply \eqref{eq38} to the rest terms of the sum. Then, we obtain
$$
\Phi_a(x) \le\frac{H(1)}{\sqrt{\pi}}\, \frac{x}{\sqrt{\ln x}\, H_a(z)}+r_4+r_3+r_2+r_1,
$$
where
$$
r_4\ll_{m_0}\frac{x}{(\ln x)^{3/2}H_a(z)}.
$$
Let us find the asymptotic formula for $1/H_a(z)$. For this purpose we apply equality \eqref{eq32}. For prime $p$ we set
$$
g_0(p) = \begin{cases}
g(p), &\text{if } p\nmid a,
\\
0 &\text{otherwise},
\end{cases}
\qquad h_0(p) = \frac{g_0(p)}{1-g_0(p)},
$$
and define $g_0(d)$, $h_0(d)$ for square-free $d$ by multiplicativity. From the definition of the function $h_0$ we obtain
$$
H_a(z) = \sum_{k\le z}\mu^2(k)h_0(k).
$$
Now we find the parameters $A_1$, $A_2$, $L$, $\kappa$ corresponding to the function $g_0$. It follows from \eqref{eq22} that one can take $A_1 = 4$. Estimate \eqref{eq21} for $p\ge 2$ implies that
$$
g(p) = \frac{1}{p}+\frac{\theta}{p(p-1)},\qquad |\theta| \le 1.
$$
From here we get $\kappa = 1$ and $A_2$, $L = O_a(1)$. From \eqref{eq32} we obtain
\begin{align*}
\frac{1}{H_a(z)} &= \prod_{p}\biggl(1-g_0(p)\biggr) \biggl(1-\frac{1}{p}\biggr)^{-1}\frac{1}{\ln z}\biggl(1+O\biggl(\frac{1}{\ln z}\biggr)\biggr)
\\
&=\frac{C\beta(a)}{\ln z}\biggl(1+O\biggl(\frac{1}{\ln z}\biggr)\biggr),
\end{align*}
where
\begin{gather*}
C = \prod_{p}(1-g(p))\biggl(1-\frac{1}{p}\biggr)^{-1},
\\
\beta(a) = \prod_{p\,|\,a}\biggl( 1+\frac{1}{p(p-1)\ln(p/(p-1))-1}\biggr).
\end{gather*}
Hence, we obtain the following estimate for the sum $\Phi_a(x)$:
\begin{equation}
\label{eq39}
\Phi_a(x) \le K\beta(a)\frac{x}{\sqrt{\ln x}\,\ln z}\biggl(1+O\biggl(\frac{1}{\ln z}\biggr)\biggr)+r_4+r_3+r_2+r_1,
\end{equation}
where
\begin{align*}
K &= \frac{H(1)C}{\sqrt{\pi}} = \frac{1}{\sqrt{\pi}}\prod_p\sqrt{p^2-p}\, \ln{\frac{p}{p-1}}\biggl(1-\frac{1}{p(p-1)\ln(p/(p-1))}\biggr) \biggl(1-\frac{1}{p}\biggr)^{-1}
\\
&= \frac{1}{\sqrt{\pi}}\prod_{p}\sqrt{\frac{p}{p-1}}\biggl(p\ln\frac{p}{p-1} -\frac{1}{p-1}\biggr).
\end{align*}
It remains to estimate the remainder terms. Now we take 
$$
f(q) = 3^{\omega(q)},\qquad B = \frac{5}{2},\qquad d = 1,\qquad Q = \frac{\sqrt{x}}{(\ln x)^A},\qquad N = x.
$$
We choose $m_0 = 5$ and
$$
z = \sqrt{Q} = \frac{\sqrt[4]{x}}{(\ln x)^{A/2}},
$$
then Lemma \ref{l13} implies that $r_2\ll x/(\ln x)^{5/2}$. For the remainders $r_4$, $r_3$, $r_1$, we obtain
$$
r_4\ll\frac{x}{(\ln x)^{5/2}},\qquad r_3\ll\frac{x}{(\ln x)^3},\qquad r_1\ll\sqrt[4]{x}.
$$
It follows from \eqref{eq39} that
$$
\Phi_a(x)\le 4K\beta(a)\frac{x}{(\ln x)^{3/2}}\biggl(1+O\biggl(\frac{\ln\ln x}{\ln x}\biggr)\biggr),
$$
which completes the proof of Theorem~\ref{t1}.

\end{fulltext}

\goodbreak 

\end{document}